\documentclass[12pt]{amsart}

\usepackage{amsmath, amssymb}
\usepackage{color, hyperref}
\usepackage[usenames,dvipsnames]{xcolor}
\usepackage{stmaryrd}

\newcommand{\bdot}{\boldsymbol{\cdot}}

\theoremstyle{plain}
\newtheorem{theorem}{\bf Theorem}[section]
\newtheorem{proposition}[theorem]{\bf Proposition}
\newtheorem{lemma}[theorem]{\bf Lemma}

\theoremstyle{definition}

\newtheorem{examples}[theorem]{\bf Examples}

\newcommand{\N}{\mathbb N}
\newcommand{\Z}{\mathbb Z}

\newcommand{\Q}{\mathbb Q}

\newcommand{\red}{{\text{\rm red}}}

 \DeclareMathOperator{\ord}{ord}

 \DeclareMathOperator{\supp}{supp}

\newcommand{\LK}{\llbracket}
\newcommand{\RK}{\rrbracket}

\renewcommand{\t}{\, | \,}

\renewcommand{\bdot}{\boldsymbol{\cdot}}

\numberwithin{equation}{section}

\hypersetup{
	pdftitle={Half-factoriality},
	pdfmenubar=false,
	pdffitwindow=true,
	pdfstartview=FitH,
	colorlinks=true,
	linkcolor=blue,
	citecolor=green,
	urlcolor=cyan
}

\makeatletter
\newcommand\zeu@Scale{1.05}
\makeatother

\begin{document}
\title{On half-factoriality of transfer Krull monoids}

\author[W. Gao]{Weidong Gao}
\address{Weidong Gao, Center for Combinatorics, LPMC-TJKLC Nankai University, Tianjin 300071, P.R. China}\email{wdgao@nankai.edu.cn}

\author[C. Liu]{Chao Liu}
\address{Chao Liu, Department of Mathematics and Statistics, Brock University, St. Catharines, Ontario, Canada L2S 3A1} \email{chaoliuac@gmail.com}

\author[S. Tringali]{Salvatore Tringali}
\address{salvatore tringali, College of Mathematics and Computer Science, Hebei Normal University, Shijiazhuang 050024, P.R. China}\email{salvo.tringali@gmail.com}

\author[Q. Zhong]{Qinghai Zhong}
\address{Qinghai Zhong\\
University of Graz, NAWI Graz \\
Institute for Mathematics and Scientific Computing \\
Heinrichstra{\ss}e 36\\
8010 Graz, Austria} \email{qinghai.zhong@uni-graz.at}

\thanks{The last-named author was supported by the Austrian Science Fund (FWF Project P28864--N35).}



\begin{abstract}
Let $H$ be a  transfer Krull monoid over a subset $G_0$ of an abelian group $G$ with finite exponent.  Then every non-unit $a\in H$ can be written
	as a finite product of atoms, say $a=u_1 \cdot \ldots \cdot u_k$. The set $\mathsf L(a)$ of all possible factorization lengths $k$ is
	called the set of lengths of $a$, and $H$ is said to be half-factorial if $|\mathsf L(a)|=1$ for all $a\in H$. 
	
	We show that, if $a \in H$ and $|\mathsf L(a^{\lfloor (3\exp(G) - 3)/2 \rfloor})| = 1$, then the smallest divisor-closed submonoid of $H$ containing $a$ is half-factorial.
	In addition, we prove that, if $G_0$ is finite and $\bigl|\mathsf L\bigl(\prod_{g\in G_0}g^{2\ord(g)}\bigr)\bigr| = 1$, then $H$ is  half-factorial.
\end{abstract}

\maketitle
\thispagestyle{empty}	
	
\section{Introduction}
	
Let $H$ be a monoid. If an element $a \in H$ has a factorization $a = u_1 \cdot \ldots \cdot u_k$, where $k\in \N$ and $u_1, \ldots, u_k$ are atoms of $H$,  then $k$ is called a factorization length of $a$, and the set $\mathsf L (a)$ of all possible $k$ is referred to as the set of lengths of $a$. The monoid $H$ is said to be half-factorial (half-factorial) if 
$|\mathsf L (a)|=1$ for every $a \in H$. Half-factoriality has been a central topic in factorization theory since the early days of this field (e.g., see \cite{Ch-Co00,  Co05a,  Ma-Ok09a, Ro11a, Co-Sm11a, Ge-Ka-Re15a, Ma-Ok16a}).

Given $a\in H$, let $\LK a\RK = \{ b \in H \mid b \ \text{divides some power of} \ a \}$ be the smallest divisor-closed submonoid of $H$ containing $a$. Then $\LK a\RK$ is half-factorial if and only if $|\mathsf L(a^n)|=1$ for all $n\in\N$, and  $H$ is half-factorial if and only if $\LK c\RK$ is half-factorial for every $c \in H$. It is thus natural to ask:

\vskip 0.2cm
\begin{quote}
Does there exist an integer $N \in \mathbb N$ such that, if $a\in H$ and $|\mathsf L(a^N)|=1$, then 
		$\LK a \RK$ is half-factorial\,? (Note that, if $\LK a \RK$ is half-factorial for some $a \in H$, then of course $|\mathsf L(a^k)| = 1$ for every $k \in \mathbb N$.)
\end{quote}
\vskip 0.1cm
We answer this question  affirmatively for transfer Krull monoids over finite abelian groups, and we study the smallest $N$ having the above property (Theorems \ref{t1} and \ref{main}).

Transfer Krull monoids and transfer Krull domains are a recently introduced class of monoids and domains including, among others, all commutative Krull domains and wide classes of non-commutative Dedekind domains (see Section \ref{2} and \cite{Ge16c} for a survey).

	Let $H$ be a transfer Krull monoid over a subset $G_0$ of an abelian group $G$. Then $H$ is half-factorial if and only if the monoid $\mathcal B (G_0)$ of zero-sum sequences over $G_0$ is half-factorial (in this case, we also say that the set $G_0$ is half-factorial). It is a standing conjecture that every abelian group has a half-factorial generating set, which implies that every abelian group can be realized as the class group of a half-factorial Dedekind domain (\cite{Ge-Go03}). 
	
	Suppose now that $H$ is a commutative Krull monoid with class group $G$ and that every class contains a prime divisor. It is a classic result that $H$ is half-factorial if and only if $|G|\le 2$, and it turns out that, also for $|G| \ge 3$, half-factorial subsets (and minimal non-half-factorial subsets) of the class group $G$ play a crucial role in a variety of arithmetical questions (see  \cite[Chapter 6.7]{Ge-HK06a}, \cite{Ge-Zh16a}). However, we are far away from a good understanding of  half-factorial sets in finite abelian groups (see \cite{Sc05c} for a survey, and \cite{Pl-Sc05a, Pl-Sc05b, Sc06a}). To mention one open question, the maximal size of half-factorial subsets is unknown even for finite cyclic groups (\cite{Pl-Sc05b}). Our results open the door to a computational approach to the study of half-factorial sets.

More in detail, denote by $\mathsf{hf} (H)$  the infimum of all $N \in \N$ with the following property:
\vskip 0.2cm 
\begin{center}
If $a \in H$ and $|\mathsf L(a^N)| = 1$, then $\llbracket a\RK$ is half-factorial.
\end{center}
\vskip 0.1cm
(Here, as usual, we assume $\inf \emptyset = \infty$.) We call $\mathsf{hf} (H)$ the \emph{half-factoriality index} of $H$.
	If $H$ is not  half-factorial, then $\mathsf{hf}(H)$ is the infimum of all  $N\in \N$ with the property that $|\mathsf L(a^N)|\ge 2$ for every $a\in H$ such that $\LK a\RK$ is not half-factorial.

\smallskip
\begin{theorem}\label{t1}
Let	 $H$ be a transfer Krull monoid over a finite subset $G_0$ of an abelian group $G$ with finite exponent.
The following are equivalent:
\begin{enumerate}
\item[(a)] $H$ is half-factorial. \item[(b)] $\mathsf{hf} (H)=1$.
\item[(c)] $G_0$ is half-factorial.
\item[(d)] $\bigl|\mathsf L\bigl(\prod_{g\in G_0}g^{2\ord(g)}\bigr)\bigr| = 1$.
\end{enumerate}	
\end{theorem}

We observe that in general if $H$ is half-factorial, then $\mathsf{hf} (H)=1$. But if $H$ is a transfer Krull monoid over a subset of a torsion free group, then $\mathsf{hf} (H)=1$ does not imply that $H$ is half-factorial (see Example \ref{2.4}.1). Furthermore, for every $n\in \N$, there exists a Krull monoid $H$ with finite class group such that 
$\mathsf{hf}(H)=n$ (see Example \ref{2.4}.2).


%

\smallskip
\begin{theorem}\label{main}
Let $H$ be a transfer  Krull monoid over an abelian group $G$. 
\begin{enumerate}
\item $\mathsf{hf} (H) < \infty$ if and only if $\exp (G) < \infty$.  
\item If $\exp(G)<\infty$ and $|G|\ge 3$, then
$
\exp(G)\le \mathsf{hf} (H)\le \frac{3}{2}(\exp(G)-1)$.
\item If $G$ is cyclic or $\exp (G) \le 6$, then $\mathsf{hf} (H)=\exp(G)$.
\end{enumerate}
\end{theorem}
%
%
We postpone the proofs of Theorems \ref{t1} and \ref{main} to Section \ref{sec:3}.

\section{Preliminaries} \label{2}

Our notation and terminology are consistent with \cite{Ge-HK06a}. Let  $\mathbb N$ be  the set of positive integers, let $\N_0=\N\cup\{0\}$, and let $\Q$ be the set of rational numbers.  For integers $a, b \in \mathbb Z$, we denote by
$[a, b ] = \{ x \in \mathbb Z \mid a \le x \le b\}$  the discrete, finite interval between $a$ and $b$.

\smallskip
\noindent
{\bf Atomic monoids.}  By a {monoid}, we mean an associative semigroup with identity, and if not stated otherwise we use multiplicative notation. Let $H$ be a monoid with identity $1=1_H \in H$.  An element $a \in H$ is said to be invertible (or a unit) if there exists an element $a'\in H$ such that $aa'=a'a=1$. The set of invertible elements of $H$ will be denoted by $H^{\times}$, and we say that $H$ is reduced if $H^{\times}=\{1\}$.  The monoid $H$ is said to be unit-cancellative if for any two elements $a,u \in H$,   each of the  equations $au=a$ or $ua=a$ implies that $u \in H^{\times}$. Clearly, every cancellative monoid is unit-cancellative.

Suppose that $H$ is unit-cancellative. An element $u \in H$ is said to be irreducible (or an atom) if $u \notin H^{\times}$ and for any two elements $a, b \in H$, $u=ab$ implies that $a \in H^{\times}$ or $b \in H^{\times}$. Let  $\mathcal A (H)$ denote the set of atoms of $H$. We say that $H$ is atomic if every non-unit is a finite product of atoms. If $H$ satisfies the ascending chain condition on principal left ideals and on principal right ideals, then $H$ is atomic  (\cite[Theorem 2.6]{Fa-Tr18a}). If $a \in H \setminus H^{\times}$ and $a=u_1 \ldots u_k$, where $k \in \N$ and $u_1, \ldots, u_k \in \mathcal A (H)$, then $k$ is a factorization length of $a$, and
\[
\mathsf L_H (a) = \mathsf L (a) = \{k \in \mathbb N \mid k \ \text{is a factorization length of} \ a \}
\]
denotes the {set of lengths} of $a$. It is convenient to set $\mathsf L (a) = \{0\}$ for all $a \in H^{\times}$.

A \emph{transfer Krull mononid} is a monoid $H$ having a weak transfer homomorphism (in the sense of \cite[Definition 2.1]{Ba-Sm15}) $\theta \colon H \to \mathcal B (G_0)$, where $\mathcal B (G_0)$ is the monoid of zero-sum sequences over a subset $G_0$ of an abelian group $G$. If $H$ is a commutative Krull monoid with class group $G$ and $G_0 \subset G$ is the set of classes containing prime divisors, then there is a weak transfer homomorphism $\theta \colon H \to \mathcal B (G_0)$. Beyond that, there are wide classes of non-commutative Dedekind domains having a weak transfer homomorphism to a monoid of zero-sum sequences (\cite[Theorem 1.1]{Sm13a}, \cite[Theorem 4.4]{Sm19a}). We refer to \cite{Ge16c, Ge-Zh19a} for  surveys on transfer Krull monoids. If $\theta \colon H \to \mathcal B (G_0)$ is a weak transfer homomorphism, then sets of lengths in $H$ and in $\mathcal B (G_0)$ coincide (\cite[Lemma 2.7]{Ba-Sm15}) and thus the statements of Theorems \ref{t1} and \ref{main} can be proved in the setting of monoids of zero-sum sequences.

\medskip
\noindent
{\bf Monoids of zero-sum sequences.} Let $G$ be an abelian group and let $G_0 \subset G$ be a non-empty subset. Then $\langle G_0 \rangle$ denotes the subgroup generated by $G_0$. In Additive Combinatorics, a { sequence} (over $G_0$) means a finite unordered sequence of terms from $G_0$ where repetition is allowed, and (as usual) we consider sequences as elements of the free abelian monoid with basis $G_0$. Let
\[
S = g_1 \cdot \ldots \cdot g_{\ell} = \prod_{g \in G_0} g^{\mathsf v_g (S)} \in \mathcal F (G_0)
\]
be a sequence over $G_0$. We call
\[
\begin{aligned}
\supp (S)&  = \{g \in G \mid \mathsf v_g (S) > 0 \} \subset G \ \text{the \ \emph{support} \ of \ $S$} \,,\\
|S|  &= \ell = \sum_{g \in G} \mathsf v_g (S) \in \mathbb N_0 \
\text{the \ \emph{length} \ of \ $S$} \,,   \\
\sigma (S) & = \sum_{i = 1}^{\ell} g_i \ \text{the \ \emph{sum} \ of \
	$S$} \,, \\
\ \ \text{ and }\ \ \  \Sigma (S) &= \Big\{ \sum_{i \in I} g_i
\mid \emptyset \ne I \subset [1,\ell] \Big\} \ \text{ the \ \emph{set of
		subsequence sums} \ of \ $S$} \,.
\end{aligned}
\]
The sequence $S$ is said to be
\begin{itemize}
	\item \emph{zero-sum free} \ if \ $0 \notin \Sigma (S)$,
	
	\item a \emph{zero-sum sequence} \ if \ $\sigma (S) = 0$,
	
	\item a \emph{minimal zero-sum sequence} \ if it is a nontrivial zero-sum
	sequence and every proper  subsequence is zero-sum free.
\end{itemize}
The set of zero-sum sequences $\mathcal B (G_0) = \{S \in \mathcal F (G_0) \mid \sigma (S)=0\} \subset \mathcal F (G_0)$ is a submonoid, and the set of minimal zero-sum sequences is the set of atoms of $\mathcal B (G_0)$.
For any arithmetical invariant $*(H)$ defined for a monoid $H$, we write $*(G_0)$ instead of $*(\mathcal B (G_0))$. In particular, $\mathcal A (G_0) = \mathcal A (\mathcal B (G_0))$ is the set of atoms of $\mathcal B (G_0)$ and $\mathsf{hf}(G_0)=\mathsf{hf}(\mathcal B(G_0))$.

\medskip

Let $G$ be an abelian group. We denote by $\exp(G)$ the exponent of $G$ which is the least common multiple of the orders of all elements of $G$.
Let $r \in \N$ and let  $(e_1, \ldots, e_r)$ be an $r$-tuple of elements of $G$. Then $(e_1, \ldots, e_r)$ is said to be independent if $e_i \ne 0$ for all $i \in [1,r]$ and if for all $(m_1, \ldots, m_r) \in \Z^r$ an equation $m_1e_1+ \ldots + m_re_r=0$ implies that $m_ie_i=0$ for all $i \in [1,r]$. Suppose $G$ is finite. The $r$-tuple $(e_1, \ldots, e_r)$ is said to be a basis of $G$ if it is independent and $G = \langle e_1 \rangle \oplus \ldots \oplus \langle e_r \rangle$. For every $n \in \N$, we denote by $C_n$ an additive cyclic group of order $n$.
 Since $G \cong C_{n_1} \oplus \ldots \oplus C_{n_r}$,  $r = \mathsf r (G)$ is the rank of $G$ and $n_r=\exp(G)$ is the exponent of $G$.

Let $G_0\subset G$ be a non-empty subset.  For a  sequence $S = g_1 \cdot \ldots \cdot g_{\ell} \in \mathcal F (G_0)$, we call
\[
\begin{aligned}
\mathsf k (S) & = \sum_{i=1}^l \frac{1}{\ord (g_i)} \in \mathbb Q_{\ge 0} \quad \text{the \emph{cross number} of $S$, and } \\
\mathsf K (G_0) & = \max \{ \mathsf k (S) \mid S \in \mathcal A (G_0) \} \quad \text{the \emph{cross number} of $G_0$}.
\end{aligned}
\]
For the relevance of cross numbers in  the theory of non-unique factorizations,  see \cite{Pl-Sc05b, Sc05d, Sc09c} and \cite[Chapter 6]{Ge-HK06a}.

The set $G_0$ is called
\begin{itemize}
	\item \emph{half-factorial} if  the monoid $\mathcal B (G_0)$ is half-factorial;
	
	\item \emph{non-half-factorial} if the monoid $\mathcal B (G_0)$ is not half-factorial;
	
	\item \emph{minimal non-half-factorial} if $G_0$ is not half-factorial but all its proper subsets are;
	
	\item an \emph{\textup{LCN}-set} if $\mathsf k(A)\ge 1$ for all atoms $A\in \mathcal A(G_0)$.
\end{itemize}

The following simple result (\cite[Proposition 6.7.3]{Ge-HK06a}) will be used throughout the paper without further mention.

\medskip
\begin{lemma} \label{2.3}
	Let $G$ be a finite abelian group and $G_0 \subset G$ a subset. Then
	the following statements are equivalent{\rm \,:}
	\begin{enumerate}
		\item[(a)]
		$G_0$ is half-factorial.
		
		\smallskip
		
		\item[(b)]
		$\mathsf k (U) = 1$ for every $U \in \mathcal A (G_0)$.
		
		\smallskip
		
		\item[(c)]
		$\mathsf L (B) = \{ \mathsf k (B) \}$ for every $B \in
		\mathcal B (G_0)$.
	\end{enumerate}
\end{lemma}

\begin{lemma}\label{l1}
	Let $G$ be a finite group, let $G_0\subset G$ be a subset, let $S$ be a zero-sum sequence over $G_0$, and let $A$ be a minimal zero-sum sequence over $G_0$.
	\begin{enumerate}
		\item If $\mathsf k(A)\neq 1$, then $|\mathsf L(A^{\exp(G)})|\ge 2$.
		
		\item If there exists a zero-sum subsequence  $T$ of $S$ such that $|\mathsf L(T)|\ge 2$, then $|\mathsf L(S)|\ge 2$.

		\item If $\mathsf k(A)<1$ and $\mathsf k(A)$ is minimal over all minimal zero-sum sequences over $G_0$, then 
		\[
		\left|\mathsf L\left(A^{\left\lceil\frac{\ord(g)}{\mathsf v_{g}(A)}\right\rceil}\right)\right|\ge2,\quad \text{for all } g\in \supp(A).
		\]
	\end{enumerate}
\end{lemma}
\begin{proof}
	1. Suppose $\mathsf k(A)\neq 1$ and let $A=g_1\cdot\ldots\cdot g_{\ell}$, where $\ell\in \N$ and $g_1,\ldots, g_{\ell}\in G_0$. Then
	$$A^{\exp(G)}=(g_1^{\ord(g_1)})^{\frac{\exp(G)}{\ord(g_1)}}\cdot\ldots\cdot (g_\ell^{\ord(g_{\ell})})^{\frac{\exp(G)}{\ord(g_{\ell})}},$$
	which implies that 
	\[
	\left\{\exp(G), \sum_{i=1}^{\ell}\frac{\exp(G)}{\ord(g_i)}\right\}=\{\exp(G), \exp(G)\mathsf k(A)\}\subset \mathsf L(A^{\exp(G)}). 
	\]
	It follows by $\mathsf k(A)\neq 1$ that $|\mathsf L(A^{\exp(G)})|\ge 2$.
	
	\medskip
	2. Suppose $T$ is a zero-sum subsequence of $S$ with $|\mathsf L(T)|\ge 2$.
	It follows by  $\mathsf L(S)\supset\mathsf L(T)+\mathsf L(ST^{-1})$ that $|\mathsf L(S)|\ge |\mathsf L(T)|\ge 2$.
	
	\medskip
	3. Suppose $\mathsf k(A)<1$ and $\mathsf k(A)$ is minimal over all minimal zero-sum sequences over $G_0$. Let $g\in \supp(A)$. Then there exist $s\in \N$ and  minimal zero-sum sequences $W_1,\ldots, W_s$ such that $$A^{\lceil\frac{\ord(g)}{\mathsf v_g(A)}\rceil}=g^{\ord(g)}\cdot W_1\cdot\ldots\cdot W_s\,.$$
	Since $$\mathsf k\left(A^{\left\lceil\frac{\ord(g)}{\mathsf v_g(A)}\right\rceil}\right)=\left\lceil\frac{\ord(g)}{\mathsf v_g(A)}\right\rceil\mathsf k(A)=1+\sum_{i=1}^s\mathsf k(W_i)> (1+s)\mathsf k(A)\,,$$ we have $\lceil\frac{\ord(g)}{\mathsf v_g(A)}\rceil\neq s+1$ and hence $\left|\mathsf L\left(A^{\left\lceil\frac{\ord(g)}{\mathsf v_{g}(A)}\right\rceil}\right)\right|\ge2$.
\end{proof}

\medskip

 For commutative and finitely generated monoids, we have the following result.
\begin{proposition}
	Let $H$ be a commutative unit-cancellative monoid. 
	If $H_{\red}$ is finitely generated, then
	$\mathsf{hf}(H)$ is finite.		
\end{proposition}

\begin{proof} We may assume that $H$ is reduced and not half-factorial. Suppose $H$ is finitely generated and suppose $\mathcal A(H)=\{u_1,\ldots, u_n\}$, where $n\in \N$. Set $A_0=\{\prod_{i\in I}u_i\mid \emptyset\neq I\subset [1,n]\}$. Then $A_0$ is finite and hence there exists $M\in \N$ such that $|\mathsf L(a_0^M)|\ge 2$ for all $a_0\in A_0$ with $\LK a_0\RK$ not half-factorial. 
	 Let $a\in H\setminus H^{\times}$ such that $\LK a\RK$ is not half-factorial. It suffices to show that 
	 $|\mathsf L(a^M)|\ge 2$.
	  Suppose $a=u_1^{k_1}\cdot \ldots \cdot u_n^{k_n}$, where $k_1,\ldots, k_n\in \N_0$. Set $I_0=\{i\in [1,n]\mid k_i\ge 1\}$ and $a_0=\prod_{i\in I}u_i$. Then $a_0$ divides  $a$ and $\LK a_0\RK=\LK a\RK$ is not half-factorial, whence $|\mathsf L(a_0^M)|\ge 2$ and $|\mathsf L(a^M)|\ge 2$. 
\end{proof}

If $G_0$ is a finite subset of an abelian group, then $\mathcal B (G_0)$ is finitely generated (\cite[Theorem 3.4.2]{Ge-HK06a}) and thus $\mathsf{hf} (G_0) < \infty$. We refer to \cite[Sections 3.2 and 3.3]{F-G-K-T17} and \cite{Ge-Re18d} for semigroups of ideals and semigroups of modules that are finitely generated unit-cancellative but not necessarily cancellative.

\begin{examples}\label{2.4}
The following examples will help up to illustrate some important points.
\begin{enumerate}
	\item Let $(e_1,e_2)$ be a basis of  $\Z^2$ and let $G_0=\{e_1, -e_1, e_2,-e_2, e_1+e_2, -e_1-e_2\}$. Then 
	$\mathcal A(G_0)=\{e_1(-e_1), e_2(-e_2), (e_1+e_2)(-e_1-e_2), e_1e_2(-e_1-e_2), (-e_1)(-e_2)(e_1+e_2)\}$.
	Since $e_1(-e_1)\bdot e_2(-e_2)\bdot (e_1+e_2)(-e_1-e_2)= e_1e_2(-e_1-e_2)\bdot (-e_1)(-e_2)(e_1+e_2)$, we obtain $G_0$ is not half-factorial. Furthermore, we have $G_1$ is half-factorial for every nonempty proper subset $G_1\subsetneq G_0$. Let $A\in \mathcal B(G_0)$. If $\supp(A)=G_0$, then $|\mathsf L(A)|\ge 2$ and $\LK A\RK=\mathcal B(G_0)$ is not half-factorial. If $\supp(A)\subsetneq G_0$, then $\LK A\RK=\mathcal B(\supp(A))$ is half-factorial and
	$|\mathsf L(A)|=1$. Therefore $\mathsf{hf}(G_0)=1$.
	
\item Let $G$ be a cyclic group with order $n$ and let $g\in G$ with $\ord(g)=n$, where $n\in \N_{\ge 3}$. Set $G_0=\{g,-g\}$. Then $G_0$ is not half-factorial. Let $A_0=g(-g)$. Then $\LK A_0\RK$ is not half-factorial and $|\mathsf L(A_0^{n-1})|=1$, whence $\mathsf{hf}(G_0)\ge n$.
Let $A\in \mathcal B(G_0)$ with $\LK A\RK$ is not half-factorial. Then $\supp(A)=G_0$  and $A_0$ divides $A$, whence $|\mathsf L(A^n)|\ge 2$. Therefore $\mathsf{hf}(G_0)=n$. Let $G\cong C_2^2$ and let $(e_1,e_2)$ be a basis of $G$. Set $G_1=\{e_1,e_2,e_1+e_2\}$. Then $G_1$ is not half-factorial. Let $A_1=e_1e_2(e_1+e_2)$. Then $\LK A_1\RK$ is not half-factorial and $|\mathsf L(A_1)|=1$, whence $\mathsf{hf}(G_1)\ge 2$.
Let $A\in \mathcal B(G_1)$ with $\LK A\RK$ is not half-factorial. Then $\supp(A)=G_1$  and $A_1$ divides $A$, whence $|\mathsf L(A^2)|\ge 2$. Therefore $\mathsf{hf}(G_1)=2$.

	\item Let $H$ be a bifurcus moniod (i.e. $2\in \mathsf L(a)$ for all $a\in H\setminus (H^{\times}\cup \mathcal A(H))$). For examples, see \cite[Examples 2.1 and 2.2]{AA}. Since for every $a\in H\setminus H^{\times}$, we have $\{2,3\}\subset \mathsf L(a^3)$, it follows that $\mathsf{hf}(H)\le 3$ and $\mathsf{hf}(H)$ is the minimal integer $t\in \N$ such that $|\mathsf L(a^t)|\ge 2$ for all $a\in H\setminus H^{\times}$.
	 Therefore
 $\mathsf{hf}(H)=3$ if and only if there exists $a_0\in \mathcal A(H)$ such that $\mathsf L(a_0^2)=\{2\}$.

 	\item Let $H\subset F=F^{\times}\times [p_1,\ldots, p_s]$ be a non-half factorial finitely primary monoid of rank $s$ and exponent $\alpha$ (see \cite[Definition 2.9.1]{Ge-HK06a}). For every $a=\epsilon p_1^{t_1}\ldots p_s^{t_s}\in F$, we define $|\!| a|\!|=t_1+\ldots +t_s$, where $t_1,\ldots, t_s\in \N_0$ and $\epsilon\in F^{\times}$. Let $a\in H\setminus H^{\times}$. Since $H$ is primary, we have $H=\LK a\RK$ is not half-factorial. Thus $\mathsf{hf}(H)$ is the minimal integer $t\in \N$ such that $|\mathsf L(a^t)|\ge 2$ for all $a\in H\setminus H^{\times}$. Suppose $a_0\in H$ with $|\!|a_0|\!|=\min\{|\!|a|\!|\colon a\in H\setminus H^{\times}\}$. Then $a_0\in \mathcal A(H)$ and $\mathsf L(a_0^2)=\{2\}$, whence $\mathsf{hf}(H)\ge 3$.
 	
 	 If $H\setminus H^{\times}=(p_1\ldots p_s)^{\alpha}F$ and $s\ge 2$, then $H$ is bifurcus and hence $\mathsf{hf}(H)=3$.
 	Suppose $s=1$ and $H\setminus H^{\times}=(p_1)^{\alpha}F$. Let $b=\epsilon p^{\beta}\in H$. Then $p^{3\alpha}$ divides $b^4$. It follows by $p^{3\alpha}=(p^{\alpha})^3=p^{\alpha+1}p^{2\alpha-1}$ that $|\mathsf L(b^4)|\ge 2$, whence $\mathsf{hf}(H)\le 4$. If $3\beta\ge 4\alpha$, then $p^{3\alpha}$ divides $b^3$ and hence $|\mathsf L(b^3)|\ge 2$. If $3\beta\le 4\alpha-2$, then $b$ is an atom and $b^3=\epsilon^3p^{2\alpha-1}p^{3\beta-(2\alpha-1)}$, whence $|\mathsf L(b^3)|\ge 2$. If $3\beta=4\alpha-1$, then $\mathsf L(b^3)=\{3\}$. Put all together, if $\alpha\equiv 1\bmod  3$, then $\mathsf{hf}(H)=4$. Otherwise $\mathsf{hf}(H)=3$.
\end{enumerate}	
	
\end{examples}

\section{Proof of main theorem}\label{sec:3}

\begin{proposition}\label{p1}
Let $G_0\subset G$ be a non half-factorial subset and let $S$ be a zero-sum sequence over $G_0$ with $\supp(S)=G_0$.
\begin{enumerate}
\item If $G_0$ is an LCN-set, then $|\mathsf L(\prod_{g\in G_0}g^{\ord(g)})|\ge 2$.

\item If $|G_0|=2$, then $|\mathsf L(\prod_{g\in G_0}g^{\ord(g)})|\ge 2$.

\item If $G_0$ is a minimal non half-factorial subset, then $|\mathsf L(S^{\exp(G)})|\ge 2$.

\item If $|\{g\in G_0\mid \ord(g)/\mathsf v_g(S) =\exp(G)\}|\le 1$, then $|\mathsf L(S^{\exp(G)})|\ge 2$.
\end{enumerate}
\end{proposition}
\begin{proof}
	1. Suppose $G_0$ is an LCN-set. Since $G_0$ is not half-factorial, there exists a minimal zero-sum sequence $T$ over $G_0$ such that $\mathsf k(T)>1$. Note that $T$ is a subsequence of $\prod_{g\in G_0}g^{\ord(g)}$. Then there exits $W_1,\ldots, W_l\in \mathcal A(G_0)$ such that
	$$\prod_{g\in G_0}g^{\ord(g)}=T\cdot W_1\cdot \ldots\cdot W_l\,.$$
	Thus $\mathsf k(\prod_{g\in G_0}g^{\ord(g)})=|G_0|=\mathsf k(T)+\sum_{i=1}^l\mathsf k(W_i)>1+l$.  The assertion follows by $\{|G_0|, 1+l\}\subset \mathsf L(\prod_{g\in G_0}g^{\ord(g)})$.
	
	\medskip
	2. Suppose $|G_0|=2$ and let $G_0=\{g_1,g_2\}$. If $G_0$ is an
	LCN-set, the assertion follows by 1.. Suppose there exists a minimal
	zero-sum sequence $T$ over $G_0$ with $\mathsf k(T)<1$. Let
	$T_0=g_1^{l_1}\cdot g_2^{l_2}$ be the minimal zero-sum sequence over
	$G_0$ such that $\mathsf k(T_0)$ is minimal. If
	$\min\{\frac{\ord(g_1)}{l_1}, \frac{\ord(g_2)}{l_2}\}\le 2$, say
	$\frac{\ord(g_1)}{l_1}\le 2$ then
	$$T_0^2=g_1^{\ord(g_1)}\cdot W, \text{ where $W$ is non-empty zero-sum sequence}\,. $$
	Thus $\mathsf k(W)=2\mathsf k(T_0)-1<\mathsf k(T_0)$, a
	contradiction to the minimality of $\mathsf k(T_0)$. Therefore
	$\min\{\frac{\ord(g_1)}{l_1}, \frac{\ord(g_2)}{l_2}\}>2$ and hence
	$$g_1^{\ord(g_1)}\cdot g_2^{\ord(g_2)}=T_0^2\cdot V \text{ where $V$ is non-empty zero-sum sequence}\,.$$
	It follows that $|\mathsf L(g_1^{\ord(g_1)}\cdot g_2^{\ord(g_2)})|\ge 2$.
	
	\medskip
	3. Suppose that $G_0$ is a minimal non-half-factorial set. If $S$ has a minimal zero-sum subsequence $A$ with $\mathsf k(A)\neq 1$, then the assertion follows by Lemma \ref{l1}. If $G_0$ is an LCN-set, then the assertion follows from 1. and Lemma \ref{l1}.2. Therefore we can suppose $\mathsf L(S)=\{\mathsf k(S)\}$ and suppose there exists a minimal zero-sum sequence $T$ over $G_0$ with $\mathsf k(T)<1$.

	Let $T_0=\prod_{i=1}^{|G_0|}g_i^{l_i}$ be the minimal zero-sum sequence over $G_0$ such that $\mathsf k(T_0)$ is minimal. The minimality of $G_0$ implies that $l_i\ge 1$ for all $i\in [1,|G_0|]$. After renumbering if necessary, we let $$\frac{\ord(g_1)}{l_1}=\min\{\frac{\ord(g_i)}{l_i}\mid i\in [1,|G_0|]\}\,.$$
	By Lemma \ref{l1}.3, $\left|\mathsf L\left(T_0^{\left\lceil\frac{\ord(g_1)}{l_1}\right\rceil}\right)\right|\ge2$. 
	If $T_0^{\left\lceil\frac{\ord(g_1)}{l_1}\right\rceil}$ divides $ S^{\exp(G)}$, 
	the assertion follows by Lemma \ref{l1}.2.
	Suppose $T_0^{\left\lceil\frac{\ord(g_1)}{l_1}\right\rceil}\nmid S^{\exp(G)}$.
	Let $$I=\{i\in [1,|G_0|]\mid \left\lceil\frac{\ord(g_1)}{l_1}\right\rceil l_i>\exp(G)\mathsf v_{g_i}(S)\}\,.$$
	Thus for each  $i\in I$, we have $$2\ord(g_i)>l_i\left\lceil\frac{\ord(g_i)}{l_i}\right\rceil\ge l_i\left\lceil\frac{\ord(g_1)}{l_1}\right\rceil>\exp(G)\mathsf v_{g_i}(S)\ge \exp(G)\,,$$
	which implies that $\ord(g_i)=\exp(G)$, $\mathsf v_{g_i}(S)=1$, and $\left\lceil\frac{\ord(g_1)}{l_1}\right\rceil>\frac{\ord(g_i)}{l_i}=\frac{\exp(G)}{l_i}$.  
	
	Let $i_0\in I$ such that $l_{i_0}=\max\{l_i\mid i\in I\}$. Therefore  for every $j\in [1,|G_0|]\setminus I$, we have 
	$$l_j\le \frac{\exp(G)\mathsf v_{g_j}(S)}{\left\lceil\frac{\ord(g_1)}{l_1}\right\rceil}\le \frac{\exp(G)\mathsf v_{g_j}(S)}{\frac{\exp(G)}{l_{i_0}}}=l_{i_0}\mathsf v_{g_j}(S)\,.$$
	Note that for every $i\in I$, we have $l_i\le l_{i_0}=l_{i_0}\mathsf v_{g_i}(S)$. It follows by $\mathsf v_{g_{i_0}}(T_0)=l_{i_0}=l_{i_0}\mathsf v_{g_{i_0}}(S)=\mathsf v_{g_{i_0}}(S^{l_{i_0}})$ that 
	$$S^{l_{i_0}}=T_0\cdot W, \text{ where $W$ is a zero-sum sequence over $G_0\setminus\{g_{i_0}\}$} \,.$$
	By the minimality of $G_0$, we have $G_0\setminus\{g_{i_0}\}$ is half-factorial which implies that $\mathsf k(W)\in \N$. Therefore  $\mathsf k(T_0)=l_{i_0}\mathsf k(S)-\mathsf k(W)$ is an integer, a contradiction to $\mathsf k(T_0)<1$.
	
	\medskip
	4. Let $G_1=\{g\in G_0\mid \ord(g)=\exp(G)\mathsf v_g(S)\}$. Suppose $G_0\setminus G_1$ is not half-factorial. If $G_0\setminus G_1$ is an LCN-set, then the assertion follows by Proposition \ref{p1}.1 and Lemma \ref{l1}.2. Otherwise there exits a minimal zero-sum sequence $A$ over $G_0\setminus G_1$ such that $\mathsf k(A)<1$.  We may assume that $\mathsf k(A)$ is minimal over all minimal zero-sum sequences over $G_0\setminus G_1$ and  that $\min\{\frac{\ord(g)}{\mathsf v_g(A)}\mid g\in \supp(A)\}=\frac{\ord(g_0)}{\mathsf v_{g_0}(A)}$ for some $g_0\in \supp(A)\subset G_0\setminus G_1$.
	Thus by Lemma \ref{l1}.3, we have $|\mathsf L\left(A^{\left\lceil\frac{\ord(g_0)}{\mathsf v_{g_0}(A)}\right\rceil}\right)|\ge 2$. The definition of $G_1$ implies that $$A^{\left\lceil\frac{\ord(g_0)}{\mathsf v_{g_0}(A)}\right\rceil}\quad \text{ divides }\quad S^{\exp(G)}$$ and hence the assertion follows.
	
	Suppose $G_0\setminus G_1$ is  half-factorial. Then $G_1$ is non-empty and hence $G_1=\{g_0\}$ for some $g_0\in G_0$. If $G_0$ is an LCN-set, then the assertion follows by Proposition \ref{p1}.1 and Lemma \ref{l1}.2. Otherwise there exits a minimal zero-sum sequence $A$ over $G_0$ such that $\mathsf k(A)<1$.  We may assume that $\mathsf k(A)$ is minimal over all minimal zero-sum sequences over $G_0$  and  that $\min\{\frac{\ord(g)}{\mathsf v_g(A)}\mid g\in \supp(A)\}=\frac{\ord(g_1)}{\mathsf v_{g_1}(A)}$ for some $g_1\in \supp(A)\subset G_0$.
	Thus by Lemma \ref{l1}.3, we have $|\mathsf L\left(A^{\left\lceil\frac{\ord(g_1)}{\mathsf v_{g_1}(A)}\right\rceil}\right)|\ge 2$.
	For every $g\in G_0\setminus G_1$, we obtain
	$$\mathsf v_g(A)\left\lceil\frac{\ord(g_1)}{\mathsf v_{g_1}(A)}\right\rceil\le \mathsf v_g(A)\left\lceil\frac{\ord(g)}{\mathsf v_g(A)}\right\rceil<2\ord(g)\le \exp(G)\mathsf v_g(S)\,.$$
	If $\mathsf v_{g_0}(A)\left\lceil\frac{\ord(g_1)}{\mathsf v_{g_1}(A)}\right\rceil\le \ord(g_0)=\exp(G)$, then  $$A^{\left\lceil\frac{\ord(g_1)}{\mathsf v_{g_1}(A)}\right\rceil}\quad \text{ divides }\quad S^{\exp(G)}$$ and hence $|\mathsf L(S^{\exp(G)})|\ge 2$.
	
	Otherwise for every $g\in G\setminus G_1$, we have  $$\frac{\exp(G)}{\mathsf v_{g_0}(A)}<\left\lceil\frac{\ord(g_1)}{\mathsf v_{g_1}(A)}\right\rceil\le \left\lceil\frac{\ord(g)}{\mathsf v_{g}(A)}\right\rceil\le \left\lceil\frac{\exp(G)\mathsf v_{g}(S)}{2\mathsf v_{g}(A)}\right\rceil\le \frac{\exp(G)\mathsf v_{g}(S)}{\mathsf v_{g}(A)}\,.$$
	Therefore $\mathsf v_g(A) <\mathsf v_{g_0}(A)\mathsf v_g(S)$ for all $g\in G_0\setminus G_1$ which implies that $A$ divides $S^{\mathsf v_{g_0}(A)}$.
	Thus there exits a zero-sum sequence $W$ over $G_0\setminus G_1$ such that
	$S^{\mathsf v_{g_0}(A)}=A\cdot W$. Since $G_0\setminus G_1$ is half-factorial, we obtain $\mathsf k(A)=\mathsf v_{g_0}(A)\mathsf k(S)-\mathsf k(W)$ is an integer, a contradiction to $\mathsf k(A)<1$.
\end{proof}

\bigskip
\begin{proof}[Proof of Theorem \ref{t1}] By the definition of transfer Krull monoid,  it suffices to prove the assertions for $H=\mathcal B(G_0)$ and hence $H$ is half-factorial if and only if $G_0$ is half-factorial.
	If $G_0$ is half-factorial, it is easy to see that $\mathsf{hf}(G_0)=1$ and $\bigl|\mathsf L\bigl(\prod_{g\in G_0}g^{2\ord(g)}\bigr)\bigr|=1$. Therefore we only need to show that (b) implies (c) and that (d) implies (c).
	
	(b) $\Rightarrow$ (c)	Suppose $\mathsf{hf}(G_0)=1$ and assume to the contrary that $G_0$ is not half-factorial. Then there exists $A\in \mathcal A(G_0)$ such that $\mathsf k(A)\neq 1$, whence $\supp(A)$ is not half-factorial. Therefore $\mathsf{hf}(\supp(A))\ge 2$, a contradiction.

	(d) $\Rightarrow$ (c) Suppose $|\mathsf L(\prod_{g\in G_0}g^{2\ord(g)})|=1$ and
	assume to the contrary that $G_0$ is not half-factorial. If $G_0$ is an LCN set, then Proposition \ref{p1}.1 implies that  $|\mathsf L(\prod_{g\in G_0}g^{\ord(g)})|\ge 2$, a contradiction. Thus there exists an atom $A\in \mathcal A(G_0)$ with $\mathsf k(A)<1$ and we may assume that $\mathsf k(A)$ is minimal over all atoms of $\mathcal B(G_0)$. Let $g_0\in \supp(A)$. Then by Lemma \ref{l1}.3, we have  $\left|\mathsf L\left(A^{\left\lceil\frac{\ord(g_0)}{\mathsf v_{g_0}(A)}\right\rceil}\right)\right|\ge2$, a contradiction to $A^{\left\lceil\frac{\ord(g_0)}{\mathsf v_{g_0}(A)}\right\rceil}\t \prod_{g\in G_0}g^{2\ord(g)}$.
\end{proof}

\bigskip
 \begin{proof}[Proof of Theorem \ref{main}]  By the definition of transfer Krull monoid, it sufficient to prove all assertions for $H=\mathcal B(G)$.
 	
 1.  Suppose $\exp(G)<\infty$. If $|G|\ge 3$, then 2. implies that $\mathsf{hf}(G)<\infty$. If $|G|\le 2$, then $\mathcal B(G)$ is half-factorial and hence $\mathsf{hf}(G)=1$.
 
 Suppose $\exp(G)=\infty$. If there exists an element $g\in G$ with $\ord(g)=\infty$, then $A_n=((n+1)g)(-ng)(-g)$ is an atom for every $n\in \N$. Since $\{(n+1)g, -ng, -g\}$ is not half-factorial  and $|\mathsf L(A_n^n)|=1$ for every $n\ge 2$, we obtain that
 $\mathsf{hf}(G)\ge n$ for every $n\ge 2$, that is, $\mathsf{hf}(G)=\infty$.
 Otherwise $G$ is torsion. Then there exists a sequence $(g_i)_{i=1}^{\infty}$ with $g_i\in G$ and $\lim_{i\rightarrow \infty}\ord(g_i)=\infty$. It follows by 1. that
 $\mathsf{hf}(G)\ge \mathsf{hf}(\langle g_i\rangle)\ge \ord(g_i)$ for all $i\in \N$, that is,  $\mathsf{hf}(G)=\infty$.

 2. %
 If $G$ is an elementary $2$-group and $e_1, e_2$ are two independent elements, then $\{e_1,e_2,e_1+e_2\}$ is not a half-factorial set and $|\mathsf L(e_1e_2(e_1+e_2))|=1$ which implies that $\mathsf{hf}(G)\ge 2=\exp(G)$. Otherwise there exists an element $g\in G$ with $\ord(g)=\exp(G)\ge 3$. Since $\{g,-g\}$ is not half-factorial and $|\mathsf L(g^{\ord(g)-1}(-g)^{\ord(g)-1})|=1$, we obtain $\mathsf{hf}(G)\ge \ord(g)=\exp(G)$.

  Let $S$ be a zero-sum sequence over $G$ such that $\supp(S)$ is not half-factorial. In order to prove $\mathsf{hf}(G)\le \left\lfloor\frac{3\exp(G)-3}{2}\right\rfloor$,   we show that $$|\mathsf L(S^{\left\lfloor\frac{3\exp(G)-3}{2}\right\rfloor})|\ge 2\,.$$
Set $G_0=\supp(S)$. If $G_0$ is an LCN-set, the assertion follows by Proposition \ref{p1}.1.
  Suppose there exists an atom $A\in \mathcal A(G_0)$ with $\mathsf k(A)<1$.
  Let $A_0\in \mathcal A(\supp(S))$ be such that $\mathsf k(A_0)$ is minimal over all minimal zero-sum sequences over $G_0$ and set $A_0=g_1^{l_1}\cdot\ldots \cdot g_y^{l_y}$, where $y, l_1,\ldots l_y\in \N$ and $g_1,\ldots, g_y\in \supp(S)$ are pairwise distinct elements.
  If there exists $j\in [1,y]$ such that $2l_j\ge \ord(g_j)$, then $g_j^{\ord(g_i)}$ divides $A_0^2$ and hence $A_0^2=g_j^{\ord(g_j)}\cdot W$ for some non-empty  sequence $W\in \mathcal B(\supp(S))$. Thus $\mathsf k(W)=2\mathsf k(A_0)-1<\mathsf k(A_0)$, a contradiction to the minimality of $\mathsf k(A_0)$. Therefore
  $$2l_i\le  \ord(g_i)-1 \text{ for all $i\in [1,y]$}\,.$$
 After renumbering if necessary,  we assume $\frac{\ord(g_1)}{l_1}=\min\{\frac{\ord(g_i)}{l_i}\mid i\in [1,y]\}$.
  Then $$l_i\left\lceil\frac{\ord(g_1)}{l_1}\right\rceil\le l_i\left\lceil\frac{\ord(g_i)}{l_i}\right\rceil\le l_i\frac{\ord(g_i)+l_i-1}{l_i}\le \ord(g_i)+\frac{\ord(g_i)-1}{2}-1\le \frac{3\exp(G)-3}{2}\,,$$
 which implies  $A_0^{\left\lceil\frac{\ord(g_1)}{l_1}\right\rceil}$ divides $S^{\left\lfloor\frac{3\exp(G)-3}{2}\right\rfloor}$. The assertion follows by Lemma \ref{l1}.3.

   \medskip
  3(a). Suppose that  $G$ is cyclic and that $g\in G$ with $\ord(g)=|G|\ge 3$. We will show that $\mathsf{hf}(G)=\exp(G)$.
  
  Let $S$ be a zero-sum sequence over $G$ such that $\supp(G)$ is not half-factorial. It suffices to show $|\mathsf L(S^{\exp(G)})|\ge 2$.
  If $|\{g\in \supp(S)\mid \ord(g)=|G|\mathsf v_g(S)\}|\le 1$, then the assertion follows from Proposition \ref{p1}.4. Suppose $|\{g\in \supp(S)\mid \ord(g)=|G|\mathsf v_g(S)\}|\ge 2$. Then there exist
  distinct $g_1, g_2\in \supp(S)$ such that $\ord(g_1)=\ord(g_2)=|G|$. We may assume that $g_1=kg_2$ for some $k\in \N_{\ge 2}$ with $\gcd(k,|G|)=1$. It follows by $\mathsf k(g_1^{|G|-k}\cdot g_2)<1$ that
  $G_0=\{g_1,g_2\}=\{g_1,kg_1\}$ is not half-factorial. By Proposition \ref{p1}.2, we obtain that $|\mathsf L(S^{\exp(G)})|\ge 2$.

  \medskip
  3(b). Suppose $G$ is a finite abelian group with $\exp(G)\le 6$. We need to prove that  $\mathsf{hf}(G)=\exp(G)$.
   Let $S$ be a zero-sum sequence over $G$ such that $\supp(G)$ is not half-factorial. It suffices to show $|\mathsf L(S^{\exp(G)})|\ge 2$.

  If $\supp(S)$ is an LCN-set, the assertion follows by Proposition \ref{p1}.1.
 Thus there is a minimal zero-sum sequence $W$ over $\mathsf {supp}(S)$ such that \[\mathsf {k}(W)<1\,.\] By  Proposition \ref{p1}.2 and Lemma \ref{l1}.2, we have
  $$|\mathsf {supp} (W)|\geq 3.$$
Suppose $W\t S$. Since $\mathsf {k} (W)<1$, it follows by Lemma \ref{l1} that $|\mathsf L(W^{\exp(G)})|\ge 2$ and hence $|\mathsf L(S^{\exp(G)})|\ge 2$.
Therefore we may assume that $W\nmid S$, whence  $|W|\ge |\supp(W)|+1\ge 4$. It follows that  $6\ge \exp(G)\ge \frac{|W|}{\mathsf k(W)}>|W|\ge4$.

  We distinguish two cases according to $\exp(G)\in\{5,6\}.$

{\bf Case 1.} $\exp (G)=5$.

Then, $G\cong C_5^r$ and for all $W\in \mathcal A(\supp(S))$ with $\mathsf k(W)<1$, we have that 
 \[W\nmid S, \ |\supp(W)|=3,\  \text{ and } \ |W|=4\,.\]

Let $W_0$ be an atom over $\supp(S)$ with $\mathsf k(W_0)<1$. Then 
 $W_0$ and  $S$ must be of the forms
 $$W_0=g_1^2g_2g_3,\quad S=Tg_1g_2g_3$$ where $g_1, g_2, g_3\in \supp(S)$ are pairwise distinct and $T\in \mathcal F(\supp(S)\setminus \{g_1\})$ with $\sigma(T)=g_1$.

We may assume $T$ is zero-sum free. Otherwise $T=T_0T'$ with $T_0$ is zero-sum and $T'$ is zero-sum free and we can replace $S$ by $T'g_1g_2g_3$, since $|\mathsf L((T'g_1g_2g_3)^5)|\ge 2$ implies that $|\mathsf L(S^5)|\ge 2$. Therefore $S$ is a product of at most three atoms and every term of $S$ has order $5$.

Assume to the contrary that $|\mathsf L(S^5)|=1$, i.e., $\mathsf L(S^5)=\{|T|+3\}$. 
Since $g_1^5g_2^5g_3^5=W_0^2(g_1g_2^3g_3^3)$ is a zero-sum subsequence of $S^5$, we obtain that $g_1g_2^3g_3^3$ is an atom.
Note that 
\[(g_1^2g_2g_3)^2S=(g_1^4T)(g_1g_2^3g_3^3)\text{ is a zero-sum subsequence of }S^5\,.\]

Suppose $S$ is an atom. Then $\mathsf L(g_1^4T)=\{2\}$ and hence $|g_1^4T|\ge 2\times 4=8$, i.e., $|T|\ge 4$.
It follows by 
 $\{5\}=\mathsf L(S^5)=\{|T|+3\}$ that $|T|=2$, a contradiction. 

Suppose $S$ is a product of two atoms. Then $\mathsf L(g_1^4T)=\{3\}$ and hence $|g_1^4T|\ge 3\times 4=12$, i.e., $|T|\ge 8$.
It follows by 
$\{10\}=\mathsf L(S^5)=\{|T|+3\}$ that $|T|=7$, a contradiction.

Suppose $S$ is a product of three atoms. Then $\mathsf L(g_1^4T)=\{4\}$ which implies that $T=T_1T_2T_3T_4$ such that $g_1T_i$ is zero-sum for all $i\in [1,4]$. Since $g_1T_i\t S$, we obtain $\mathsf k(g_1T_i)\ge 1$ and hence $|g_1T_i|\ge 5$. Therefore $|g_1^4T|\ge 4\times 5=20$, i.e., $|T|\ge 16$.
It follows by 
$\{15\}=\mathsf L(S^5)=\{|T|+3\}$ that $|T|=12$, a contradiction. 

{\bf Case 2.} $\exp(G)=6$.

Let $W$ be an atom over $\supp(S)$ with $\mathsf k(W)<1$.
If $|W|=4$,  then  $|\mathsf {supp}(W)|= 3$ and hence 
$W$ must be of the form
$$W=g_1^2g_2g_3$$ where $g_1,g_2,g_3\in \supp(S)$ are pairwise distinct. Since $(g_1^6)(g_2^6) (g_3^6)=W^3 (g_2^3g_3^3)$, we obtain that 
$|\mathsf L(g_1^6g_2^6g_3^6)|\ge 2$. It follows by $g_1^6g_2^6g_3^6$ divides $S^{\exp(G)}$ that $|\mathsf L(S^{\exp(G)})|\ge 2$.

\smallskip
If $|W|=5$, then  $|\mathsf {supp}(W)|= 3$ or $4$ and hence 
$W$ must be one of the following
forms.

\begin{enumerate}
\item[i.] $W=g_1^3g_2g_3$ with $g_1,g_2,g_3\in \supp(S)$ are pairwise distinct.
\item[ii.] $W=g_1^2g_2^2g_3$ with $g_1,g_2,g_3\in \supp(S)$ are pairwise distinct.
\item[iii.] $W=g_1^2g_2g_3g_4$ with $g_1,g_2,g_3,g_4\in \supp(S)$ are pairwise distinct.
\end{enumerate}

Suppose (i) holds. Then $0=2\sigma(W)=6g_1+2g_2+2g_3=2g_2+2g_3$. Since $(g_1^6)(g_2^6)(g_3^6)=W^2 (g_2^2g_3^2)^2$, we obtain that 
$|\mathsf L(g_1^6g_2^6g_3^6)|\ge 2$. It follows by $g_1^6g_2^6g_3^6$ divides $S^{\exp(G)}$ that $|\mathsf L(S^{\exp(G)})|\ge 2$.

Suppose (ii) holds.  Then $0=3\sigma(W)=6g_1+6g_2+3g_3=3g_3$. Thus $\ord(g_3)=2$ and hence $\mathsf k(W)\ge 1/2+4/6>1$, a contradiction.

Suppose
(iii) holds.  Then $0=3\sigma(W)=6g_1+3g_3+3g_3+3g_4=3g_2+3g_3+3g_4$. Therefore
$W_0=g_2^3g_3^3g_4^3$ is zero-sum. 
 If $W_0=g_1^6g_2^6g_3^6g_4^6(W)^{-3}$
is not a minimal zero-sum sequence,  then $|\mathsf L(g_1^6g_2^6g_3^6g_4^6)|\ge 2$ and hence  $|\mathsf L(S^{\exp(G)})|\ge 2$.
If $W_0$ is minimal
zero-sum, then $g_1^6g_2^6g_3^6g_4^6=(g_1^6)W_0^2$ implies that $|\mathsf L(g_1^6g_2^6g_3^6g_4^6)|\ge 2$ and hence  $|\mathsf L(S^{\exp(G)})|\ge 2$.
\end{proof}

\providecommand{\bysame}{\leavevmode\hbox to3em{\hrulefill}\thinspace}
\providecommand{\MR}{\relax\ifhmode\unskip\space\fi MR }
\providecommand{\MRhref}[2]{%
  \href{http://www.ams.org/mathscinet-getitem?mr=#1}{#2}
}
\providecommand{\href}[2]{#2}

\end{document}